\documentclass{elsarticle}

\usepackage{lineno,hyperref}
\usepackage{amssymb}
\usepackage[utf8]{inputenc}
\usepackage[english]{babel}
\usepackage{todonotes}
\usepackage{amsmath} 
\usepackage{amsthm}
\usepackage{xcolor}
\usepackage{tikz}

\usetikzlibrary{automata,positioning,decorations.text,topaths,arrows.meta,decorations.pathmorphing,quotes}
\tikzstyle{every picture}+=[remember picture,inner xsep=0,inner ysep=0.5ex]
\modulolinenumbers[5]
\newtheorem{theorem}{Theorem}[section]
\newtheorem{lemma}[theorem]{Lemma}
\newtheorem{definition}[theorem]{Definition}
\newtheorem{example}[theorem]{Example}

\newtheorem{pro}[theorem]{Proposition}
\newtheorem{remark}[theorem]{Remark}
\newtheorem{Q}[theorem]{Question}

\journal{.}

\newcommand{\z}{\mathbb{Z}}

\begin{document}

\begin{frontmatter}

\title{Relation between finite topological spaces and finitely presentable groups}

\author{Samuel Rold\'an, Jose Luis Mora, Edward Becerra}
\address{Mathematics Department\\ Universidad Nacional de Colombia}
\fntext[myfootnote]{ This research did not receive any specific grant from funding agencies in the public, commercial, or not-for-profit sectors.}



\address{Carrera 30 calle 45, Ciudad Universitaria, Bogot\'a, Colombia}

\begin{abstract}
In this paper it is shown how to construct a finite topological space $X$ for a given finitely presentable group $G$ such that $\pi_1(X)\cong G$. Our construction is not optimal in the sense that the cardinality of the space $X$ might not be the smallest possible. Our main result applies to a large class of interesting groups, including all finite groups.
\end{abstract}

\begin{keyword}
\texttt{Finite topological space, Finitely presented group, Properly discontinuous actions, Simplicial complex, Fundamental group}
\end{keyword}

\end{frontmatter}

\section{Introduction}
Finite topological spaces are naturally related via a combinatorial focus, based on the famous bijective correspondence between topologies defined on a finite set $X$ and preorders on $X$, as shown by Alexandroff \cite{alexandroff}. Under this bijection, partial orders on $X$ correspond to $T_0$ topologies on $X$. Moreover, there is a weak homotopy equivalence between any finite $T_0$–space and the geometric realization of its associated partial ordered set \cite{mccord1966singular}. These results allow the study of properties of the order complex of a finite partially ordered set by means of its associated finite $T_0$–space. The latter has proved to be very useful \cite{Barmak}. Indeed, the theory of finite topological spaces gives interesting tools to study partial ordered sets and polyhedra.

From a naive viewpoint, the topology defined on a finite space could seem uninteresting. Nevertheless, for a finite set $X$ and a fixed topology defined on it, we can define continuous functions from an interval $I$ to $X$ in a very natural way. Namely, we can consider paths, loops and homotopies on $X$. This paper begins with an effort to understand how the classical tools from algebraic topology translate to finite topological spaces. Later on, more interesting questions are considered relating the topological structure of $X$ with the classical topological invariants introduced in algebraic topology. Namely, for a finite topological space $X$ we consider its fundamental group $\pi_1(X)$. The group $\pi_1(X)$ does not have to be a finite group. We can find some interesting examples where a given topology on $X$ leads to $\pi_1(X)\cong \z$. This leads us to the main topic in question:
\begin{Q}\label{question1}
Let $G$ a finitely presentable group. Is it possible to find some finite topological space $X$ such that $\pi_1(X)\cong G$? 
\end{Q}
Finitely presentable groups appear in many topics in topology and algebra. For instance, the fundamental group of any compact surface (i.e. a closed surface without frontier) is a finitely presentable group. In fact, any finitely presentable group is the fundamental group of a compact  4-dimensional space-time manifold. On the other side, the theory of finitely presentable groups is largely related to coding theory. We might look for such connections between our result and those applied fields in future papers.    

This paper is organized as follows: In Section 2 we describe some tools from algebraic topology to be used throughout Sections 3 and 4, in order to justify our developments. In Section 3 we introduce an interesting construction of a finite topological space associated to a given finitely presentable group $G$. Then the main result is presented, Theorem \ref{main}, which can be applied to a large class of groups, including finite groups. Also in Section 3, another construction of a finite topological space associated to the finite abelian group $\z_n$ for $n\in \mathbb{N}$ is shown. This construction has the advantage of enabling the presentation of explicit examples, because the commutativity of the group $\z_n$ provides a simpler construction than the one presented in Section 2.

The authors want to express their gratitude to Alexis Sanchez and Roberto Vargas for contributing to this paper with stimulating discussions about Hasse diagrams. The third author wants to thank the first and second authors for their efforts in order to clarify the initial idea behind this paper. 

\section{Preliminaries}
In this section, in order to obtain an easier explanation of the results presented in later sections, some tools from algebraic topology are shown. These tools are standard in a classical course on this topic. Our  main references are the book in General Topology by J. Munkres \cite{munkres} Chapters 9-12, and Chapter 1 in \cite{Hatcher}.

Let $X$ be a path-connected topological space. We consider $\pi_1(X,x_0)$ the set of \emph{loops} in $X$ based on a fixed point $x_0\in X$, up to homotopy. This set is actually a group, called the \emph{fundamental group} of the space $X$,  with the composing operation defined for $[\alpha]$ and $[\beta]$ two classes in $\pi_1(X,x_0)$, via $[\alpha][\beta]:=[\alpha*\beta]$, where $*$ denotes concatenation (\cite{munkres}, Chapter 9, \S 51) of loops based on $x_0$. It is an easy exercise to prove that for a path-connected space $X$ the group $\pi_1(X,x_0)$ does not depend on the point $x_0$, so we can write $\pi_1(X)$. A topological space is called simply connected if its fundamental group is the trivial group.

\subsection{Group actions}
Let $X$ be a topological space, and let $G$ be group. An \emph{action} from the group $G$ onto the space $X$ is a map:
\begin{equation}
    \tau:G\times X\rightarrow X \text{   for each  }g\in G
\end{equation}
such that for each $g\in G$, the map $\tau|_{\{g\}\times X}:=\tau_g:X\rightarrow X$ is a homeomorphism and satisfies:
\begin{itemize}
    \item[i.] $\tau_e=id_X$ for $e$ the identity element in $G$.
    \item[ii.] $\tau_g(\tau_h(x))=\tau_{gh}(x)$ for all $g,h\in G$ and any $x\in X$.
\end{itemize}
 For short, we set $\tau_g(x)=gx$. The \emph{orbit space} $X/G$ is defined to be the quotient space obtained from $X$ by means of the equivalence relation $x\sim gx$ for all $x\in X$ and all $g\in G$. The equivalence class of $x$ is called the \emph{orbit} of $x$.
 \begin{example}
 Let $X$ be a set. Suppose that we have an action from a group $G$ on $X$. We can define an action from the group $G$ on the set $\wp(X)$ (parts of $X$) via the formula
\begin{equation}\label{acinparts}
g(A)=\{ga\in X \mid a\in A\}\mbox{   for }g\in G, A\subset X.    
\end{equation}
It is a straightforward verification that the action defined in Equation \ref{acinparts} satisfies the conditions to be a well defined action on $\wp(X)$. 
 \end{example}
\begin{definition}
If $G$ is a group of homeomorphisms of $X$, the action of $G$ on $X$ is said to be properly discontinuous if for every $x\in X$ there is a neighborhood $U$ of $x$ such that $g(U)$ is disjoint from $U$ whenever $g\not= e$. ($e$ is the identity element of $G$.)

\end{definition}
A topological space $X$ is defined to be locally path-connected if for each point $x\in X$ there exists a neighbourhood $U_x$ of $x$ such that $U_x$ is path-connected.
\begin{theorem}\label{theo:MainHatcher}[\cite{Hatcher}, Ch.1 Proposition 1.40 Part (c)] Let X be a simply connected and locally path-connected topological space endowed with an action from a group $G$. If the action is provided to be properly discontinuous, the group $G$ is isomorphic to the group $\pi_1(X/G)$.

\end{theorem}



\section{Construction for a finitely presentable group $G$}\label{sections}

Let $G$ be a finitely presentable group with presentation $\langle S \mid R\rangle$ where $S$ is a non empty finite set and $R$ a finite subset of $F_S$, the free group generated by $S$. By definition, every $r\in R$ is a word of $F_S$ with $length(r)$ letters in the alphabet $S$. Let us denote $r[i]$ the $i-$th letter of $r$, $r[0]=e$ be the identity element of $F_S$ and $r[i,j]=r[i]r[i+1]\cdots r[j]$  for $0\leq i\leq j\leq length(r)$. Note that there is a homomorphism $\varphi:F_S\rightarrow G$ such that $Ker(\varphi)=\langle R\rangle$ and $G\cong F_S/\langle R\rangle$; therefore, two elements $w,w'\in F_S$ will represent the same element in $G$ if $\varphi(w)=\varphi(w')$. This will be denoted $w=_{_G} w'$.

\begin{theorem}
Let G be a finitely presentable group. There exists a reduced presentation $\langle S \mid R\rangle$ in the sense that for every $s\in S$ and every $r\in R$, $s\neq_Ge$ and $r[i,j]=_{_G}e$ only for $i=0,1$ and $j=length(r)$.

\end{theorem}
\begin{proof}
 The case when there is a $s\in S$ such that $s=_{_G} e$ is trivial, and if there exists $r\in R$ and $n<m$ such that $r[n,m]=_{_G} e$, then we replace $R$ by $$\left(R-\{r\}\right)\cup \{r[n,m],r[1]\cdots r[n-1]r[m+1]\cdots r[length(r)]\}$$
\end{proof}

For each presentation $\langle S \mid R\rangle$ we can construct its Cayley complex $X_G$. To do this we will perform an inductive construction. Let the vertices or 0-cells be the elements of G; these will form the 0-skeleton of the complex $X_{G,0}$. Inductively, we will attach (n+1)-cells to the n-skeleton $X_{G,n}$ with attachment maps $\gamma_n:S^n\rightarrow X_{G,n}$. For each element $(g,s)\in G\times S$ we put a 1-cell connecting the vertices $g$ and $gs$. Now, for every element $(g,r)\in G\times R$ we attach a 2-cell to the path given by the edges denoted by  $$(gr[0,0],r[1])(gr[0,1],r[2])\cdots(gr[0,length(r)-1], r[length(r)])$$
or equivalently, the path that passes through the vertices
$$gr[0,0],gr[0,1],\ldots, gr[0,length(r)].$$

\begin{remark}
\label{remark:SigmaG}
We will denote by $\Sigma_G $ the set of cells of the Cayley complex $X_G$. We can identify the 0-cells with $G$, the 1-cells with $G\times S$ and the 2-cells with $G\times R$; therefore, we will denote the elements of $\Sigma_G$ with elements of $G\cup (G\times S)\cup(G\times R)$.
\end{remark}

We want the closure of every cell to be homeomorphic to a closed ball in  some $\mathbb{R}^n$, or equivalently, for the attaching maps to be embeddings. By construction, it is sufficient that the attaching maps are injective. This means that each 1-cell has non overlapping endpoints and the closed path in the boundary of each 2-cell does not self-intersect. Without loss of generality, we can assume that we are working with a reduced presentation. With such presentation, the edge $(g,s)$ has endpoints $g,gs\in G$ and $g=_{_G}gs$ would imply the contradiction $s=_{_G}e$. If the closed path in the boundary of a 2-cell $(g,r)$ self-intersects, it has to self intersect at a vertex; this would mean that $gr[0,n]=_{_G}gr[0,m]$ and therefore $r[n+1,m]=_{_G}e$ for $n\neq 0,1$ or $m\neq length(r)$ which would be a contradiction.

The previous condition on the Cayley complex allows us to construct a simplicial complex performing a generalized barycentric subdivision (\cite{lundell2012topology}, pp. 80-82). Let $\tilde{\Sigma}_G\subseteq \wp (\Sigma_G)$ be the simplicial complex with n-simplexes corresponding to the non-empty chains with the order $\preccurlyeq$ given by the condition \textit{``is in the boundary of''} in $\Sigma_G$. This way our simpleces have n-cells of $X_G$ such that for every pair of them, one is in the boundary of the other; as an example, for every $r\in R$ there is a 2-simplex $\{(e,r),(e,r[1]),e\}\in\tilde{\Sigma}_G$. We want $\tilde{\Sigma}_G$ to be a topological space. With this in mind, we endow $\tilde{\Sigma}_G$ with the Alexandrov topology induced by inclusion; that is, the topology $\tau=\{S\subseteq\tilde{\Sigma}_G:\forall \sigma,\sigma'\in S\, \ (\sigma\in S\ \land\ \sigma'\subseteq \sigma)\ \rightarrow\ \sigma' \in S  \}$ with base $\{star(\sigma):\sigma\in \tilde{\Sigma}_G\}$ and $star(\sigma)=\{\lambda\in\tilde{\Sigma}_G:\sigma\subseteq \lambda\}$, the usual definition for simplicial complexes.

\begin{lemma}\label{lem:connectednes}
$\tilde{\Sigma}_G$ is path connected, simply connected and locally path connected.
\end{lemma}

\begin{proof}
We know that there is a weak homotopy equivalence between $\tilde{\Sigma}_G$ and its geometric realization $|\tilde{\Sigma}_G|$ (\cite{mccord1966singular}, Theorem 2). Also, we know that the latter is homeomorphic to $X_G$ as $\tilde{\Sigma}_G$ is the generalized barycentric subdivision of $X_G$. Therefore, there is a weak homotopy equivalence between $X_G$ and $\tilde{\Sigma}_G$. In particular, 
$$0=\pi_0(X_G,x_0)=\pi_0(\tilde{\Sigma}_G,\sigma_0),$$
$$0=\pi_1(X_G,x_0)=\pi_1(\tilde{\Sigma}_G,\sigma_0),$$
which means that $\tilde{\Sigma}_G$ is path connected and simply connected. We also have that a space with an Alexandrov topology is  locally path connected (\cite{mccord1966singular}, Corollary of lemma 6).
\end{proof}

Now, we want to define an action of $G$ on $\tilde{\Sigma}_G$. For this, we will define an action of $G$ over $\Sigma_G$ and for each $\sigma\in\tilde{\Sigma}_G\subset \wp(\Sigma_G)$, we define the action of $G$ with $g\sigma=\{gx:x\in\sigma\}$. Note that this is well defined as an action on $\tilde{\Sigma}_G$ if and only if  the action of $G$ over $\Sigma_G$ sends chains to chains, or equivalently, if the action of $G$ over $\Sigma_G$ preserves the order induced by inclusion. Every action of $G$ on a set $A\subseteq \wp(\Sigma_G)$ that can be defined this way will be called an induced action.

\begin{lemma}
\label{Theo:PropDisc}
Let $G$ act on $\tilde{\Sigma}_G\subseteq\wp(\Sigma_G)$ as an induced action, $\sigma\in\tilde{\Sigma}_G$ and $g\in G$ different from the identity of $G$. If the action of $g$ on $\Sigma_G$ does not have fixed points and the action of $g$ on $\Sigma_G$ preserves the dimension of the n-cells, then $star(\sigma)\cap g (star(\sigma))=\phi$.
\end{lemma}

\begin{proof}
As $\sigma$ is non-empty, let $\lambda\in\sigma$ so that $star(\sigma)\subseteq star(\{\lambda\})$ and therefore $star(g\sigma)\subseteq star(\{g\lambda\})$. If we assume that there exists $\alpha\in star(\sigma)\cap star(g\sigma)$, then $\alpha\in star(\{\lambda\})\cap star(\{g\lambda\})$. Thus $\lambda, g\lambda\in\alpha$ but we have that $\lambda\neq g \lambda$ as the action of $g$ on $\Sigma_G$ does not have fixed points. Now, as $\alpha$ is a chain, $\lambda$ is in the boundary of $g\lambda$ or $g\lambda$ is in the boundary of $\lambda$ which would give us a contradiction as $\lambda$ and $g\lambda$ are n-cells of the same dimension.
\end{proof}

We define an action of an element $g\in G$ on $\Sigma_G$ for each n-cell with n=0,1,2 using Remark \ref{remark:SigmaG}.This is:

\begin{itemize}
    \item  $g\in G$ acting on the 0-cell $h\in G$ is the 0-cell $gh\in G$ 
    \item  $g\in G$ acting on the 1-cell $(h,s)\in G\times S$ is the 1-cell $g(h,s)=(gh,s)\in G\times S$ 
    \item  $g\in G$ acting on the 2-cell $(h,r)\in G\times R$ is the 2-cell $g(h,r)=(gh,r)\in G\times R$ 
\end{itemize}

With the induced action of $G$ over $\tilde{\Sigma}_G$, we can form a space $\tilde{\Sigma}_G/G$, the quotient space of $\tilde{\Sigma}_G$ in which each point $\sigma$ is identified with all its images $g\sigma$ as $g$ ranges
over $G$. The points of $\tilde{\Sigma}_G/G$ are thus the orbits $G\sigma= { g\sigma : g \in G }$ in $\tilde{\Sigma}_G$. 

\begin{theorem}
G is isomorphic to $\pi_1(\tilde{\Sigma}_G/G)$.
\end{theorem}
\begin{proof}
It is clear that the action of $G$ on $\Sigma_G$ we defined satisfies the hypothesis of Theorem \ref{Theo:PropDisc} for every $g\in G$ different from the group identity element. Therefore, the fact that the action is properly discontinuous follows directly from Lemma \ref{Theo:PropDisc} and that  $\{star(\sigma):\sigma\in \tilde{\Sigma}_G\}$ is a base for $\tilde{\Sigma}_G$. Now, using Lemma \ref{lem:connectednes}, we have all the hypothesis of Theorem \ref{theo:MainHatcher} (also including the path-connectedness assumed in all of Section 2) and therefore, using Theorem \ref{theo:MainHatcher}, we conclude that $G$ is isomorphic to $\pi_1(\tilde{\Sigma}_G/G)$.
\end{proof}

Now, we want to prove that $\tilde{\Sigma}_G/G$ is finite. First we will characterize every orbit of $\tilde{\Sigma}_G/G$. Let $G\sigma\in \tilde{\Sigma}_G/G$. As $\sigma$ is a non-empty finite chain, there exists $\alpha=\max\sigma$. For all $g\in G$, $G(g\sigma)=G\sigma$; therefore, without loss of generality, we can assume that $\alpha\in\{e\}\cup (\{e\}\times S)\cup(\{e\}\times R)$. If we fix $\alpha$, we have that there is only a finite number of $\beta\in\Sigma_G$ such that $\beta\preccurlyeq\alpha$ and therefore there is a finite number of $\sigma\in\tilde{\Sigma}_G/G$ with $\alpha=\max\sigma$. Furthermore, as $G$ is finitely presentable, there is only a finite amount of $\alpha\in\{e\}\cup (\{e\}\times S)\cup(\{e\}\times R)$ and we can conclude that  $\tilde{\Sigma}_G/G$ is finite.

\begin{theorem}\label{main} Let G have a finite presentation $\langle S \mid R\rangle$. Then $G\cong \pi_1(\tilde{\Sigma}_G/G)$ with
 $$|\tilde{\Sigma}_G/G|=1+3|S|+\Sigma_{r\in R}\left(4 \text{length}(r) +1\right)$$
\end{theorem}

\begin{proof} 
We will prove the last statement of the theorem. Let $G\sigma\in\tilde{\Sigma}_G/G$ and $\alpha=\max\sigma$. As we saw in the last paragraph, we can assume $\alpha\in \{e\}\cup (\{e\}\times S)\cup(\{e\}\times R)$. Therefore, there are three possibilities:

\begin{enumerate}
    \item $\alpha=e$: As $e$ is minimal in $\Sigma_G$, $\sigma=\{e\}$.
    \item $\alpha=(e,s)$ for $s\in S$: $\sigma\in\{\{\alpha\},\{\alpha,e\},\{\alpha,s\}\}$.
    \item $\alpha=(e,r)$ for $r\in R$: As we can assume $S\subset G$, we can interpret $r[0,i]\in F_G$ as an element of $G$ , and this case can be split in five:
    
    \begin{enumerate}

        \item $\sigma=\{\alpha, r[0,i]\}$ for $i=0,\cdots,length(r)-1$.
        \item $\sigma=\{\alpha, (r[0,i],r[i+1])\}$ for $i=0,\cdots,length(r)-1$.
        \item $\sigma=\{\alpha, (r[0,i],r[i+1]), r[0,i]\}$ for $i=0,\cdots,length(r)-1$.
        \item $\sigma=\{\alpha, (r[0,i],r[i+1]), r[0,i+1]\}$ for $i=0,\cdots,length(r)-1$
        \item $\sigma=\{\alpha\}$.
        
    \end{enumerate}
    
    \end{enumerate}
    and the formula $|\tilde{\Sigma}_G/G|=1+3|S|+\Sigma_{r\in R}\left(4 \text{length}(r) +1\right)$ should be clear. 
    
\end{proof}

\section{The case of $G$ a finite abelian group}\label{sectionj}
Despite the generality and beauty of the presentation described in the previous section, it is important to highlight the difficulty in finding explicit examples using the results therein. Nevertheless, we know that finite abelian groups are more handleable than non-abelian ones. We realize that if a group $G$ is an abelian group, we can find an explicit finite topological space $X$ associated to the group $G$ in the sense that $\pi_1(X)\cong G$. This is great news, since explicit examples can be considered. 

In general, for a finite abelian group $G$ it is possible to define explicitly the topology on a finite set $X$ such that $\pi_1(X)\cong G$. For the groups $\z_n\cong \z/n\z$ there is a simpler way to do it than for an arbitrary group. Furthermore, by the known property $\pi_1(X\times Y)\cong \pi_1(X)\times \pi_1(Y)$ of the fundamental group together with the Theorem of Classification of Finite Abelian Groups, it is enough to present the construction for the group $\mathbb{Z}_n$ for $n\in \mathbb{N}$.

\subsection{The construction for the group $\z_n$}
Let $G=\mathbb{Z}_n$ with $n>1$ and $n\in\mathbb{N}$. Let us consider the space $\tilde{X}_G$ consisting of $n$ disks $D_1,\ldots,D_n$ with their boundary circles identified. There is a generator $g$ of $G$ which acts on this union by sending $D_i$ to $D_{i+1}$ via a $2\pi/n$ rotation with $i$ being taken mod $n$.\\
Let $a_1,\ldots,a_{2n}$ be $2n$ equidistant points on the boundary of $\tilde{X}_G$ and ordered in the clockwise direction. Also, let $a'_1,\ldots,a'_n$ be points in $\tilde{X}_G$ such that $a'_i$ is the center of the disk $D_i$ for $1\leq i\leq n$. We denote $b_{i,1},\ldots,b_{i,2n}$ as the $2n$ open edges between $a_1,\ldots,a_{2n}$ and $a'_i$ respectively for each  $1\leq i\leq n$. Notice that the edges $b_{i,1},\ldots,b_{i,2n}$ divide the disk $D_i$ in $2n$ areas. Let us denote the open part of these $2n$ areas by $c_{i,1},\ldots,c_{i,2n}$ maintaining the coherence between these and the notation of the segments. The collection of all these elements will be denoted by $\tilde{T}_G$.
We proceed now to endow the set $\tilde{T}_G$ with a topology that makes continuous the map $F:\tilde{X}_G\rightarrow \tilde{T}_G$ which sends an element in $\tilde{X}_G$ to the element in $\tilde{T}_G$ that represents the location of the point. For instance, if $x\in\tilde{X}_G$ is located on the open edge represented by $b_{1,1}$ then $F(x)=b_{1,1}$. Thus, we will endow $\tilde{T}_G$ with the final topology with respect to the function $F$.

Now we are going to define the following equivalence relation in $\tilde{X}_G$: $x\sim y$ if both elements are located in a region represented by the same element in $\tilde{T}_G$. Then, we endow the set $\tilde{T}_G$ with the topology that makes this space homeomorphic to the quotient space $\tilde{X}_G/ \sim$. This means that the map $F=H\circ q:\tilde{X}_G\rightarrow\tilde{T}_G$ is continuous, where $q:\tilde{X}_G\rightarrow \tilde{X}_G/\sim$ is the quotient map and $H:\tilde{X}_G/\sim\ \rightarrow \tilde{T}_G$ is a homeomorphism.
Notice that the map
$$q_*:\pi_1(\tilde{X}_G)\rightarrow\pi_1(\tilde{X}_G/\sim)$$
is surjective and hence we have that $\pi_1(\tilde{T}_G)=0$. Thus, we have that $\tilde{T}_G$ is simply-connected, path-connected, locally path-connected and $G$ is a covering space action on it . Therefore, the map
$$p:\tilde{T}_G\rightarrow \tilde{T}_G/G$$ is the universal cover and hence $\pi_1(\tilde{T}_G/G)\cong G=\mathbb{Z}_n$.

Lastly, if we take into account the Fundamental Theorem  of finitely generated abelian groups and the fact that the fundamental group of a finite product of topological spaces is isomorphic to the finite product of the fundamental groups of the topological spaces, then it is possible to construct a finite topological space with fundamental group $G$ for $G$ any finitely generated abelian group.

\medskip

Now we are going to show explicitly what the open sets in the resulting topology are. For $1<n\in\mathbb{N}$ we define the sets
\begin{align*}
A' &=\{a'_1,\ldots,a'_{2n}\},&  A& =\{a_1,\ldots,a_{n}\},& B'& =\{b'_1,\ldots,b'_{2n}\},\\ 
B_i&=\{b_{i,1},\ldots,b_{i,2n}\},& C_i&=\{c_{i,1},\ldots,c_{i,2n}\},     
\end{align*}
for $1\leq i\leq n$. Let $\Sigma=A\cup A'\cup B\cup B'\cup C$, where $B=\cup_{i=1}^nB_i$ and $C=\cup_{i=1}^nC_i$. Endow the set $\Sigma$ with the following topology:
for $c_{i,j}\in C$ the basic open set will be $U_{c_{i,j}}=\{c_{i,j}\}$. For $b_{i,j}\in B$ the basic open set will be $U_{b_{i,j}}=\{b_{i,j}\}\cup U_{c_{i,j-1}}\cup U_{c_{i,j}}=\{c_{i,j-1},b_{i,j},c_{i,j}\}$ with $j$ taken mod $n$ and for $b'_i\in B'$ it will be $U_{b'_i}=\bigcup_{j=1}^nU_{c_{j,i}}\cup \{b'_i\}=\{c_{1,i},\ldots,c_{n,i},b'_i\}$. For $a_i\in A$ the basic open set will be $U_{a_i}=B_i\cup C_i\cup \{a_i\}$ and for $a'_i\in A'$ it will be $U_{a'_i}=\bigcup_{j=1}^n U_{b_{j,i}}\cup\{b'_{i-1},a'_i,b'_i\}=\{b'_{i-1},a'_i,b'_i,b_{1,i},c_{1,i-1},c_{1,i},\ldots,b_{n,i},c_{n,i-1},c_{n,i}\}$ with $j$ taken mod $n$ and $i$ taken mod $2n$. The collection of all these open sets together with the empty set form a topological basis because every intersection of any two elements is empty or contains at least one element of the form $c_{i,j}$, and the latter implies that the open set $U_{c_{i,j}}$ (which belong to the collection) is contained in the intersection.\\


Now we define the following equivalence relation: $a'_i\sim a'_j$ iff $i\equiv j\text{ mod } 2$. $a_i\sim a_j$ for all $1\leq i,j\leq n$. $b'_i\sim b'_j$ iff $i\equiv j\text{ mod } 2$. $b_{i,s}\sim b_{j,t}$ iff $2(j-i)\equiv t-s\text{ mod }2n$ and $c_{i,s}\sim c_{j,t}$ iff $2(j-i)\equiv t-s\text{ mod }2n$.\\
Let $\Sigma/\sim$ be the quotient space and $$p:\Sigma\rightarrow \Sigma/\sim,$$ the quotient map. Let $g:\Sigma\rightarrow \Sigma$ be the function defined by 
$$g(x)= \left\{ \begin{array}{lcc}
             a'_{i+2} &   if  & x=a'_i
             \\
             \\ a_{i+1} &  if & x=a_i \\
             \\ b'_{i+2} &  if & x=b'_i \\
             \\ b_{i+1,j+2} &  if & x=b_{i,j} \\
             \\ c_{i+1,j+2} &  if  & x=c_{i,j}
             \end{array}
   \right., $$
 with $i$ taken mod $n$, $j$ taken mod $2n$. We define the group $G= \langle g \rangle$ with composition of functions as the binary operation. $G$ is isomorphic to $\mathbb{Z}_n$ because $g^n(x)=x$; furthermore, the action of the group on $\Sigma$ is properly discontinuous:
\begin{align*}
     g(U_{a'_i})\cap U_{a'_i}&=U_{a'_{i+2}}\cap U_{a'_i}=\emptyset,  &g(U_{c_{i,j}})\cap U_{c_{i,j}}&=U_{c_{i+1,j+2}}\cap U_{c_{i,j}}=\emptyset,\\ 
     g(U_{b'_i})\cap U_{b'_i}&=U_{b'_{i+2}}\cap U_{b'_i}=\emptyset,& g(U_{b_{i,j}})\cap U_{b_{i,j}}&=U_{b_{i+1,j+2}}\cap U_{b_{i,j}}=\emptyset,\\ 
     g(U_{a_i})\cap U_{a_i}&=U_{a_{i+1}}\cap U_{a_i}=\emptyset.
\end{align*}

Next we are going to prove that $\Sigma$ is path-connected and locally path-connected: assume there are $U,V$ open sets with $U\not=\emptyset$ and $V\not=\emptyset$ such that $U\cup V=\Sigma$ and $U\cap V=\emptyset$. Notice that $(A\cup A')\not\subset U$ because otherwise we would have $V=\emptyset$. Similarly,  $(A\cup A')\not\subset V$. Therefore, there exist $x,y\in A\cup A'$ such that $x\in U$ and $y\in V$. Thus, $U_x\subset U$ and $U_y\subset V$, but $U_x\cap U_y\not=\emptyset$ because $x,y\in A\cup A'$. Hence, $U\cap V\not=\emptyset$ which is a contradiction. Then, $\Sigma$ is connected and also path-connected (these two invariants are equivalent in finite topological spaces).\\
To see that $\Sigma$ is locally path-connected the only thing remaining to be proved is that each one of the basic open sets is path-connected. For $x,y\in \Sigma$ we denote by $f_{y,x}:[0,1]\rightarrow U_x$ the function such that $f_{y,x}(a)=y$ for $0\leq a<1$, $f_{y,x}(1)=x$. Notice that if $U_y\subset U_x$, then $f_{y,x}$ is continuous; in fact $U_y\subset U_x$ for every $x\in \Sigma$ and $y\in U_x$ because $\Sigma$ is finite. Hence, for any $x\in \Sigma$ and $y\in U_x$ the function $f_{y,x}$ is continuous. Thus, there exists a path connecting each element $y$ in $U_x$ with $x$ and therefore $U_x$ is path-connected.

Finally, it will be proved that $\Sigma$ is simply connected in order to apply Theorem \ref{theo:MainHatcher} to show that $\mathbb{Z}_n$ is the fundamental group of $\Sigma/\sim$ and for this we need the following proposition.



\begin{pro}\label{4.1}
Let $f:I\rightarrow T$, $g:I\rightarrow T$ be two paths such that $f(0)=g(0)$, $f(1)=g(1)$ and for any $x\in I$, $f(x)\in U_{a'_i}$ and $g(x)\in U_{a'_i}$. Then $f$ and $g$ are homotopic. 
\end{pro}

\begin{proof}
From \cite{Barmak} Remark 1.2.8 we have that $U_a$ is simply connected and therefore $f$ and $g$ must be homotopic. 
\end{proof}

Let $\alpha:I\rightarrow\Sigma$ be any loop with $\alpha(0)=\alpha(1)=a'_1$. For a fixed $1\leq i\leq n$, the open set $\alpha^{-1}(U_{a_i})$ can be viewed as the countable union of disjoint open intervals. Let $(d,d')$ be one of those intervals. Notice that $\alpha(d)\in A'\cup B'$ and $\alpha(d')\in A'\cup B'$ as well. Moreover, we have that $A'\cup B'\cap U_{a_i}=\emptyset$, and therefore there is $\epsilon>0$ such that $\alpha(x)\not=a_i$ for $x\in(d,d+\epsilon)\cup (d'-\epsilon,d')$. If we restrict the function $\alpha$ to the interval $[d+\epsilon/2,d'-\epsilon/2]$ and consider it as a path, then by Proposition $\ref{4.1}$ we have that this path is homotopic to a path $\gamma$ with $a_i\not\in\operatorname{Im}(\gamma)$. Applying this process for every open interval in $\alpha^{-1}(U_{a_i})$ and every $i$, we can construct a loop $\alpha_1$ homotopic to $\alpha$ such that $A\cap\operatorname{im}(\alpha_1)=\emptyset$. Let $\alpha_2$, $\alpha_3$ be the paths defined by 
\begin{align*}
 \alpha_2(x)&= \left\{ \begin{array}{lcc}
             a'_j &   if  & \alpha_1(x)=b_{i,j}\\
             b'_j &  if  & \alpha_1(x)=c_{i,j}\\
             \alpha_1(x) &  if  & \alpha_1(x)\in A'\cup B'
             \end{array}
   \right.,& 
     \alpha_3(x)&= \left\{ \begin{array}{lcc}
             b_{1,j} &   if  & \alpha_2(x)=a'_j
             \\
             \\ c_{1,j} &  if  & \alpha_2(x)=b'_j
             \end{array}
   \right..
\end{align*}

Notice that $\alpha_2$ is indeed a path because $\alpha_2^{-1}(U_x)=\emptyset$ for $x\not\in A'\cup B'$, $\alpha_2^{-1}(U_{a'_i})=\alpha_1^{-1}(U_{a'_i})$ and $\alpha_2^{-1}(U_{b'_i})=\alpha_1^{-1}(U_{b'_i})$ for $1\leq i\leq 2n$. Similarly, $\alpha_3$ is a path because $\alpha_3^{-1}(U_{b_{1,j}})=\alpha_2^{-1}(U_{a'_j})$ and $\alpha_3^{-1}(U_{c_{1,j}})=\alpha_2^{-1}(U_{b'_j})$.

The map $H:[0,1]\times [0,1]\rightarrow T$ defined by
$$H(s,t)= \left\{ \begin{array}{lcl}
             \alpha_1(s) & \text{if} & t\in [0,1/4) \\
             \alpha_2(s) & \text{if} & t\in [1/4,1/2] \\
             \alpha_3(s) & \text{if} & t\in (1/2,3/4) \\
             a_1 & \text{if} & t\in [3/4,1]
             \end{array}
   \right. $$
is continuous: for $1\leq i\leq 2n$ we have $H^{-1}(U_{b'_i})=\alpha_1^{-1}(U_{b'_i})\times[0,3/4)$ and $H^{-1}(U_{a'_i})=\alpha_1^{-1}(U_{a'_i})\times[0,3/4)$. For $1<i\leq n$, $H^{-1}(U_{b_{i,j}})=\alpha_1^{-1}(U_{b_{i,j}})\times [0,1/4)$ and $H^{-1}(U_{c_{i,j}})=\alpha_1^{-1}(U_{c_{i,j}})\times [0,1/4)$; in the case $i=1$ we have $$H^{-1}(U_{b_{1,j}})=\left(\alpha_1^{-1}(U_{b_{1,j}})\times [0,1/4)\right)\cup \left(\alpha_3^{-1}(U_{b_{1,j}})\times (1/2,3/4)\right)$$ and $$H^{-1}(U_{c_{1,j}})=\left(\alpha_1^{-1}(U_{c_{1,j}})\times [0,1/4)\right)\cup \left(\alpha_3^{-1}(U_{c_{1,j}})\times (1/2,3/4)\right).$$
And lastly, we have 
$$H^{-1}(U_{a_1})=\left(\alpha_1^{-1}(U_{a_1})\times[0,1/4)\right)\cup\left((1/2,1]\times(1/2,1]\right).$$

Then $\alpha_1$ is null-homotopic, and therefore $\alpha$ is null-homotopic  as well because these two loops are homotopic.
\subsection{Examples}
We present some examples by constructing finite topological spaces associated to the finite groups $\z_2$, $\z_3$ and $\z_4$ respectively. A useful way to describe finite spaces is by using \emph{Hasse diagrams} (cf. \cite{Barmak} Example 1.1.1). The Hasse diagram of a partially ordered space $X$ is a graph whose vertices are the points of $X$ and whose edges are the ordered pairs $(x,y)$ such that $x <y$ and there exists no $z \in X$ such that $x < z < y$. In the planar description of a Hasse diagram we will not draw an arrow from $x$ to $y$, using a segment instead with $x$ lower than $y$. For a finite topological space $X$ and $x,y\in X$, we define the order $\preceq$ where $x\preceq y$ if and only if $y\in U$ implies $x\in U$ for all $U$ open in $X$.

\begin{itemize}
\item[I.] For $\z_2$ we have the set $\Sigma_2=\{a_1',a_2',a,b_1',b_2',b_1,\ldots,b_4,c_1,\ldots,c_4\}$ with basic open sets $U_{c_i}=\{c_i\}$ for $1\leq i\leq 4$, $U_{b_i}=\{b_i,c_{i-1},c_i\}$ for $1\leq i\leq 4$ with $i$ taken mod $4$, $U_{b'_1}=\{b'_1\}\cup\{c_i\in\Sigma_2 \mid i\text{ is odd}\}$, $U_{b'_2}=\{b'_2\}\cup\{c_i\in\Sigma_2 \mid i\text{ is even}\}$, $U_a=C\cup B\cup\{a\}$ with $C=\{c_1,\ldots,c_4\}$ and $B=\{b_1,\ldots,b_4\}$, $U_{a'_1}=\{a'_1,b'_1,b'_2\}\cup C\cup\{b_i\in\Sigma_2 \mid i\text{ is odd}\}$ and $U_{a'_2}=\{a'_2,b'_1,b'_2\}\cup C\cup\{b_i\in\Sigma_2 \mid i\text{ is even}\}$. Below is the Hasse diagram for $\Sigma_2$:

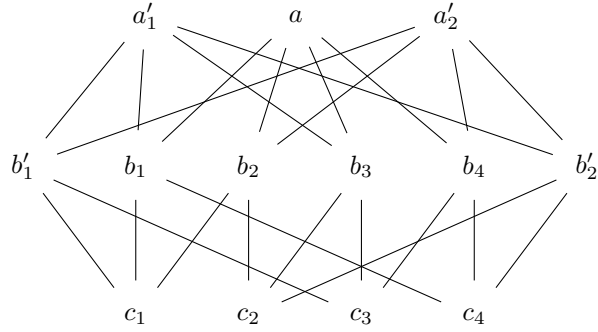
\begin{figure}[h]
    \centering
    
\begin{tikzpicture}[shorten >=1pt,node distance=1.5cm,on grid,auto]
  \tikzstyle{every state}=[draw=none,text=black,scale=1]

  \node[state]         (a2)                       {$a'_2$};
  \node[state]         (a)   [left=2cm of a2]        {$a$};
  \node[state]         (a1)  [left=2cm of a]         {$a'_1$};
  \node[state]         (b4)  [below left=2cm and .63cm of a]  {$b_2$};
  \node[state]         (b3)  [left of=b4]
  {$b_1$};
  \node[state]         (b1)  [left of=b3]         {$b'_1$};
  \node[state]         (b5)  [right of=b4]        {$b_3$};
  \node[state]         (b6)  [right of=b5]        {$b_4$};
  \node[state]         (b2)  [right of=b6]         {$b'_2$};
  \node[state]         (c1)  [below=2cm of b3]        {$c_1$};
  \node[state]         (c2)  [below=2cm of b4]        {$c_2$};
  \node[state]         (c3)  [below=2cm of b5]        {$c_3$};
  \node[state]         (c4)  [below=2cm of b6]        {$c_4$};

  \path  (a1)  edge              node {} (b1)
               edge              node {} (b2)
               edge              node {} (b3)
               edge              node {} (b5)
         (a2)  edge              node {} (b1)
               edge              node {} (b2)
               edge              node {} (b4)
               edge              node {} (b6)
         (a)   edge              node {} (b3)
               edge              node {} (b4)
               edge              node {} (b5)
               edge              node {} (b6)
         (b1)  edge              node {} (c1)
               edge              node {} (c3)
         (b2)  edge              node {} (c2)
               edge              node {} (c4)
         (b3)  edge              node {} (c1)
               edge              node {} (c4)
         (b4)  edge              node {} (c1)
               edge              node {} (c2)
         (b5)  edge              node {} (c2)
               edge              node {} (c3)
         (b6)  edge              node {} (c3)
               edge              node {} (c4);
\end{tikzpicture}
\caption{Hasse diagram for $\z_2$}
    \label{fig:my_label3}
\end{figure}

\item[II.]
For $\mathbb{Z}_3$ we have the set $\Sigma_2=\{a_1',a_2',a,b_1',b_2',b_1,\ldots,b_6,c_1,\ldots,c_6\}$ with basic open sets $U_{c_i}=\{c_i\}$ for $1\leq i\leq 6$, $U_{b_i}=\{b_i,c_{i-1},c_i\}$ for $1\leq i\leq 6$ with $i$ taken mod $6$, $U_{b'_1}=\{b'_1\}\cup\{c_i\in\Sigma_2 \mid i\text{ is odd}\}$, $U_{b'_2}=\{b'_2\}\cup\{c_i\in\Sigma_2 \mid i\text{ is even}\}$, $U_a=C\cup B\cup\{a\}$ with $C=\{c_1,\ldots,c_6\}$ and $B=\{b_1,\ldots,b_6\}$, $U_{a'_1}=\{a'_1,b'_1,b'_2\}\cup C\cup\{b_i\in\Sigma_2 \mid i\text{ is odd}\}$ and $U_{a'_2}=\{a'_2,b'_1,b'_2\}\cup C\cup\{b_i\in\Sigma_2 \mid i\text{ is even}\}$. The corresponding Hasse diagram is:

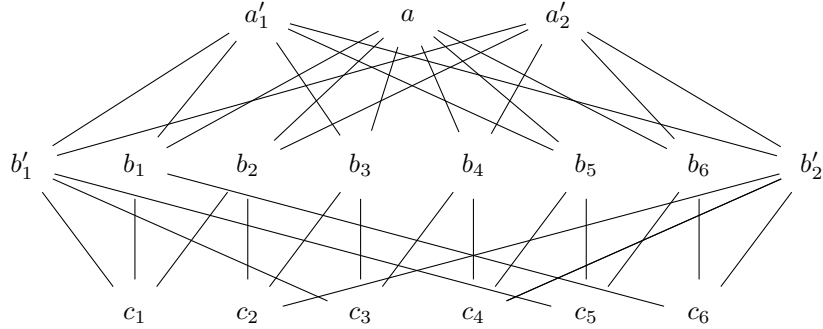
\begin{figure}[h]
    \centering
    
\begin{tikzpicture}[shorten >=1pt,node distance=1.5cm,on grid,auto]
  \tikzstyle{every state}=[draw=none,text=black,scale=1]

 \node[state]         (a2)                       {$a'_2$};
 \node[state]         (a)   [left=2cm of a2]        {$a$};
  \node[state]         (a1)  [left=2cm of a]         {$a'_1$};
  \node[state]         (b5)  [below left=2cm and .63cm of a]     {$b_3$};
  \node[state]         (b4)  [left of=b5]   
  {$b_2$};
  \node[state]         (b3)  [left of=b4]  
  {$b_1$};
  \node[state]         (b1)  [left of=b3]         {$b'_1$};
  \node[state]         (b6)  [right of=b5]        {$b_4$};
  \node[state]         (b7)  [right of=b6]        {$b_5$};
  \node[state]         (b8)  [right of=b7]        {$b_6$};
  \node[state]         (b2)  [right of=b8]         {$b'_2$};
  \node[state]         (c1)  [below=2cm of b3]        {$c_1$};
  \node[state]         (c2)  [below=2cm of b4]        {$c_2$};
  \node[state]         (c3)  [below=2cm of b5]        {$c_3$};
  \node[state]         (c4)  [below=2cm of b6]        {$c_4$};
  \node[state]         (c5)  [below=2cm of b7]        {$c_5$};
  \node[state]         (c6)  [below=2cm of b8]        {$c_6$};

  \path  (a1)  edge              node {} (b1)
               edge              node {} (b2)
               edge              node {} (b3)
               edge              node {} (b5)
               edge              node {} (b7)
         (a2)  edge              node {} (b1)
               edge              node {} (b2)
               edge              node {} (b4)
               edge              node {} (b6)
               edge              node {} (b8)
         (a)   edge              node {} (b3)
               edge              node {} (b4)
               edge              node {} (b5)
               edge              node {} (b6)
               edge              node {} (b7)
               edge              node {} (b8)
         (b1)  edge              node {} (c1)
               edge              node {} (c3)
               edge              node {} (c5)
         (b2)  edge              node {} (c2)
               edge              node {} (c4)
               edge              node {} (c6)
               edge              node {} (c4)
         (b3)  edge              node {} (c1)
               edge              node {} (c6)
         (b4)  edge              node {} (c1)
               edge              node {} (c2)
         (b5)  edge              node {} (c2)
               edge              node {} (c3)
         (b6)  edge              node {} (c3)
               edge              node {} (c4)
         (b7)  edge              node {} (c4)
               edge              node {} (c5)
         (b8)  edge              node {} (c5)
               edge              node {} (c6)       ;
\end{tikzpicture}
\caption{Hasse diagram for $\z_3$}
    \label{fig:my_label2}
\end{figure}

\item[III.] Finally, for $\mathbb{Z}_4$ we have the set $\Sigma_2=\{a_1',a_2',a,b_1',b_2',b_1,\ldots,b_8,c_1,\ldots,c_8\}$ with basic open sets $U_{c_i}=\{c_i\}$ for $1\leq i\leq 8$, $U_{b_i}=\{b_i,c_{i-1},c_i\}$ for $1\leq i\leq 8$ with $i$ taken mod $8$, $U_{b'_1}=\{b'_1\}\cup\{c_i\in\Sigma_2 \mid i\text{ is odd}\}$, $U_{b'_2}=\{b'_2\}\cup\{c_i\in\Sigma_2 \mid i\text{ is even}\}$, $U_a=C\cup B\cup\{a\}$ with $C=\{c_1,\ldots,c_8\}$ and $B=\{b_1,\ldots,b_8\}$, $U_{a'_1}=\{a'_1,b'_1,b'_2\}\cup C\cup\{b_i\in\Sigma_2 \mid i\text{ is odd}\}$ and $U_{a'_2}=\{a'_2,b'_1,b'_2\}\cup C\cup\{b_i\in\Sigma_2 \mid i\text{ is even}\}$. The corresponding Hasse diagram is:
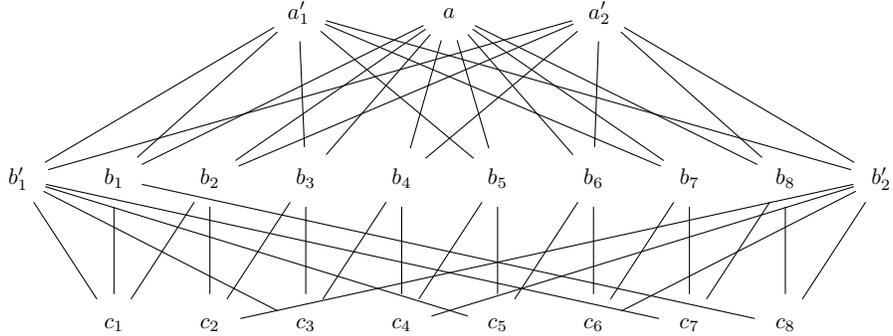
\begin{figure}[h]
    \centering
    
\begin{tikzpicture}[shorten >=1pt,node distance=1.5cm,on grid,auto]
  \tikzstyle{every state}=[draw=none,text=black,scale=.85]

  \node[state]         (a2)                       {$a'_2$};
  \node[state]         (a)   [left=2cm of a2]     {$a$};
  \node[state]         (a1)  [left=2cm of a]         {$a'_1$};
  \node[state]         (b6)  [below left=2.2cm and .63cm of a]     {$b_4$};
  \node[state]         (b5)  [left of=b6]  
  {$b_3$};
  \node[state]         (b4)  [left of=b5]         {$b_2$};
  \node[state]         (b3)  [left of=b4]         {$b_1$};
  \node[state]         (b1)  [left of=b3]         {$b'_1$};
  \node[state]         (b7)  [right of=b6]        {$b_5$};
  \node[state]         (b8)  [right of=b7]        {$b_6$};
  \node[state]         (b9)  [right of=b8]        {$b_7$};
  \node[state]         (b10) [right of=b9]        {$b_8$};
  \node[state]         (b2)  [right of=b10]         {$b'_2$};
  \node[state]         (c1)  [below of=b3, node distance=2.3cm]        {$c_1$};
  \node[state]         (c2)  [below of=b4, node distance=2.3cm]        {$c_2$};
  \node[state]         (c3)  [below of=b5, node distance=2.3cm]        {$c_3$};
  \node[state]         (c4)  [below of=b6, node distance=2.3cm]        {$c_4$};
  \node[state]         (c5)  [below of=b7, node distance=2.3cm]        {$c_5$};
  \node[state]         (c6)  [below of=b8, node distance=2.3cm]        {$c_6$};
  \node[state]         (c7)  [below of=b9, node distance=2.3cm]        {$c_7$};
  \node[state]         (c8)  [below of=b10, node distance=2.3cm]        {$c_8$};

  \path  (a1)  edge              node {} (b1)
               edge              node {} (b2)
               edge              node {} (b3)
               edge              node {} (b5)
               edge              node {} (b7)
               edge              node {} (b9)
         (a2)  edge              node {} (b1)
               edge              node {} (b2)
               edge              node {} (b4)
               edge              node {} (b6)
               edge              node {} (b8)
               edge              node {} (b10)
         (a)   edge              node {} (b3)
               edge              node {} (b4)
               edge              node {} (b5)
               edge              node {} (b6)
               edge              node {} (b7)
               edge              node {} (b8)
               edge              node {} (b9)
               edge              node {} (b10)
         (b1)  edge              node {} (c1)
               edge              node {} (c3)
               edge              node {} (c5)
               edge              node {} (c7)
         (b2)  edge              node {} (c2)
               edge              node {} (c4)
               edge              node {} (c6)
               edge              node {} (c8)
         (b3)  edge              node {} (c1)
               edge              node {} (c8)
         (b4)  edge              node {} (c1)
               edge              node {} (c2)
         (b5)  edge              node {} (c2)
               edge              node {} (c3)
         (b6)  edge              node {} (c3)
               edge              node {} (c4)
         (b7)  edge              node {} (c4)
               edge              node {} (c5)
         (b8)  edge              node {} (c5)
               edge              node {} (c6)
         (b9)  edge              node {} (c6)
               edge              node {} (c7)
         (b10) edge              node {} (c7)
               edge              node {} (c8)       ;
\end{tikzpicture}
\caption{Hasse diagram for $\z_4$}
    \label{fig:my_label}
\end{figure}
\end{itemize}

\begin{remark}
The cardinality of the finite topological space $X_{\z_n}$ constructed in this section for the group $\z_n$ is the same as in Theorem \ref{main} of the previous section applied to the group $\z_n$. However, from the discussion at the beginning of this section, the finite topological space associated to the group $\z_n\times\z_m$ by the method above turns out to be $X_{\z_n}\times X_{\z_m}$. The cardinality of the latter is greater than the cardinality provided by Theorem \ref{main} applied to the group $\z_n\times\z_m$. Actually, neither of the two constructions presented in this paper provides us with a minimal finite topological space associated to a finitely presented group $G$, in the sense of having minimal cardinality in the class of spaces $X_G$ such that $\pi(X_G)\cong G$. The Hasse diagrams presented above show that the construction described throughout Section 3 provides reduced spaces in the sense that the topology described in $X_{\z_n}$ can not be reduced by using beat points \cite{Barmak}.
\end{remark}

\bibliography{main}

\begin{thebibliography}{1}
\expandafter\ifx\csname url\endcsname\relax
  \def\url#1{\texttt{#1}}\fi
\expandafter\ifx\csname urlprefix\endcsname\relax\def\urlprefix{URL }\fi
\expandafter\ifx\csname href\endcsname\relax
  \def\href#1#2{#2} \def\path#1{#1}\fi

\bibitem{alexandroff}
P.~Alexandroff, Mat. sbornik, Diskrete Raume 2 (1937) 501--519.

\bibitem{mccord1966singular}
M.~C. McCord, et~al., Singular homology groups and homotopy groups of finite
  topological spaces, Duke Mathematical Journal 33~(3) (1966) 465--474.

\bibitem{Barmak}
J.~Barmak, \href{https://books.google.com.co/books?id=KR3zCAAAQBAJ}{Algebraic
  Topology of Finite Topological Spaces and Applications}, Lecture Notes in
  Mathematics, Springer Berlin Heidelberg, 2011.
\newline\urlprefix\url{https://books.google.com.co/books?id=KR3zCAAAQBAJ}

\bibitem{munkres}
J.~Munkres, Topology, Featured Titles for Topology Series, Prentice Hall,
  Incorporated, 2000.

\bibitem{Hatcher}
A.~Hatcher, \href{https://cds.cern.ch/record/478079}{{Algebraic topology}},
  Cambridge Univ. Press, Cambridge, 2000.
\newline\urlprefix\url{https://cds.cern.ch/record/478079}

\bibitem{lundell2012topology}
A.~T. Lundell, S.~Weingram, The topology of CW complexes, Springer Science \&
  Business Media, 2012.

\end{thebibliography}

\end{document}